\documentclass[reqno]{amsart}
\usepackage[english]{babel}
\usepackage{amsmath,amssymb}
\usepackage{amsthm,thmtools}
\usepackage{amsfonts}

\usepackage[UKenglish]{isodate}
\cleanlookdateon % if you want to be all European about it :-)
\usepackage{lmodern,microtype}
\usepackage[rgb]{xcolor}
\usepackage{enumitem}
\usepackage{tikz}
\usetikzlibrary{arrows}
\usepackage{hyperref}
\title{Quantum automorphism groups of trees}
\author[Dobben de Bruyn \and Kar \and Roberson \and Schmidt \and Zeman]{Josse van Dobben de Bruyn \and Prem Nigam Kar \and David E.~Roberson \and Simon Schmidt \and Peter Zeman}
\address{Department of Applied Mathematics and Computer Science, Technical~University~of~Denmark, 2800~Kongens~Lyngby, Denmark}
\email{jdob@dtu.dk, pkar@dtu.dk, dero@dtu.dk, pezem@dtu.dk}
\address{QMATH, Department of Mathematical Sciences, University of Copenhagen, Universitetsparken 5, 2100 Copenhagen \O, Denmark}
\email{dero@dtu.dk}
\address{Faculty of Computer Science, Ruhr University Bochum, Universitätsstra{\ss}e 150, 44801 Bochum, Germany}
\email{s.schmidt@rub.de}

\newcommand{\hair}{\ifmmode\mskip1mu\else\kern0.08em\fi} % hair space; see e.g. http://www.read.seas.harvard.edu/~kohler/latex.html
\newcommand{\myautoref}[2]{\hyperref[#2]{\autoref*{#1}\ref*{#2}}}  % \myautoref{prop:1}{itm:a} creates a link [Proposition 1(a)].

\newcounter{myenum}
\newenvironment{my-enumerate}{%
	\setcounter{myenum}{0}%
	\let\mysaveitem=\item%
	\renewcommand{\item}{%
		\par%
		\stepcounter{myenum}%
		(\Roman{myenum}).
	}%
}{%
	\let\item=\mysaveitem%
}

\tikzset{>=stealth',vertex/.style={circle,fill,inner sep=1.3pt}}
\colorlet{c0}{black}
\definecolor{mysalmon}{rgb}{.98,.5,.448}  % "Salmon" color from xcolor svgnames
\definecolor{mydodger}{rgb}{.116,.565,1}  % "Dodgerblue" color from xcolor svgnames
\colorlet{c1}{mysalmon!85!black}
\colorlet{c2}{mydodger!90!black}
\colorlet{mybg}{brown!30!gray!7}

\newcommand{\minitree}[4]{
	\begin{scope}[shift={#1}]
		\draw[#3] (0,0) node[vertex,label=#4] (R#2) {};
		\foreach[count=\tel] \x in {-.45,-.15,.45} {
			\draw (\x,-.5) node[vertex] (v\tel) {};
			\draw (R#2) -- (v\tel);
			\foreach \xx in {-.1,0,.1} {
				\draw[dash pattern=on 2pt off 1.4pt] (v\tel) -- ++(\xx,-3.3mm);
			}
		}
		\draw (.15,-.5) node[scale=.8] {$\cdots$};
	\end{scope}
}

\newcommand{\nanotree}[4]{
	\begin{scope}[shift={#1},xscale=.9]
		\draw[#3] (0,0) node[vertex,label=#4] (R#2) {};
		\foreach[count=\tel] \x in {-.3,-.1,.3} {
			\draw (\x,-.5) node[vertex] (v\tel) {};
			\draw (R#2) -- (v\tel);
			\foreach \xx in {-.07,0,.07} {
				\draw[dash pattern=on 2pt off 1.4pt] (v\tel) -- ++(\xx,-3.3mm);
			}
		}
		\draw (.1,-.5) node[scale=.6] {$\cdots$};
	\end{scope}
}

\declaretheorem[style=definition,numberwithin=section]{definition}
\declaretheorem[style=plain,numberlike=definition]{theorem}
\declaretheorem[style=plain,numberlike=definition]{lemma}
\declaretheorem[style=plain,numberlike=definition]{proposition}
\declaretheorem[style=plain,numberlike=definition]{corollary}
\declaretheorem[style=definition,numberlike=definition]{remark}

\newcommand{\mbbS}{\mathbb{S}}
\newcommand{\mbbG}{\mathbb{G}}
\newcommand{\mbbH}{\mathbb{H}}

\DeclareMathOperator{\Qut}{Qut}

\newcommand{\tensor}{\mathbin{\otimes}}

\DeclareSymbolFont{bbold}{U}{bbold}{m}{n}
\DeclareSymbolFontAlphabet{\mathbbold}{bbold}
\newcommand{\one}{\ensuremath{\mathbbold{1}}}

\DeclareFontFamily{U}{mathb}{\hyphenchar\font45}
\DeclareFontShape{U}{mathb}{m}{n}{
<5> <6> <7> <8> <9> <10> gen * mathb
<10.95> mathb10 <12> <14.4> <17.28> <20.74> <24.88> mathb12
}{}
\DeclareSymbolFont{mathb}{U}{mathb}{m}{n}
\DeclareMathSymbol{\bigast}{2}{mathb}{"06}

\colorlet{Josse_color}{orange!70!black}
\colorlet{Prem_color}{red!80!black}
\colorlet{David_color}{purple!80!black}
\colorlet{Simon_color}{blue!80!black}
\colorlet{Peter_color}{green!75!black}

\begin{document}

\begin{abstract}
	We give a characterisation of quantum automorphism groups of trees.
	In particular, for every tree, we show how to iteratively construct its quantum automorphism group using free products and free wreath products.
	This can be considered a quantum version of Jordan's theorem for the automorphism groups of trees.
	This is one of the first characterisations of quantum automorphism groups of a natural class of graphs with quantum symmetry.
\end{abstract}

\maketitle

\section{Introduction}

Attempts to characterise automorphism groups of trees can be traced back to Jordan~\cite{Jordan-center} and P\'{o}lya~\cite{polya_trees}.
It is well known that the class of automorphism groups of trees can be constructed from the trivial group using two types of group products:
the direct product and the wreath product with a symmetric group $\mathbb{S}_n$.
In this paper, we prove an analogue of this result for \emph{quantum} automorphism groups of trees.

Banica~\cite{banica_qut} and Bichon~\cite{QBic} defined the quantum automorphism group of a finite graph to be a quotient of the quantum symmetric group $\mathbb{S}_n^+$, introduced by Wang~\cite{wang}.
The theory of quantum automorphism groups of graphs is still a relatively young field and we refer to~\cite{Schmidt-dissertation} for a survey of the state of the art. Most of the known results answer the question of whether a given graph has quantum symmetry, i.e., whether the graph's quantum automorphism group differs from its automorphism group.
Explicit calculations are known only for some very special families of graphs.
These include the hypercube graph~\cite{banica_cube} and the folded hypercube graph~\cite{simon_folded_cube}.
Recently, Gromada~\cite{gromada_qut} formulated a general technique for determining the quantum automorphism groups of graphs that arise as Cayley graphs of abelian groups.
This technique can be used to calculate the quantum automorphism group of the halved hypercube graph, folded hypercube graph, and Hamming graph.

Trees are the one of the most basic graph classes. However, there have not been many attempts to calculate quantum automorphism groups of trees explicitly.
Fulton~\cite{Fulton-dissertation} proved that trees with automorphism groups isomorphic to $\mathbb{S}_2^k$, for $k\geq 2$, have quantum symmetry.
In fact, almost all graphs do not have quantum symmetry~\cite{Lupini-Mancinska-Roberson}, while almost all trees do have quantum symmetry~\cite{simon_all_trees}, which makes the study of quantum automorphism groups of trees particularly interesting.

We prove that the quantum automorphism of a tree can be constructed from the trivial quantum group using two types of products:
the free product and the free wreath product with a quantum symmetric group $\mathbb{S}_n^+$.
This is formalized by the following theorem.

\begin{theorem}\label{thm:main_theorem}
	The class $\mathcal{T}$ of all quantum automorphism groups of trees can be constructed inductively as follows:
	\begin{enumerate}[label = (\roman*)]
		\item $\one \in \mathcal{T}$.
		\item If $\mbbG, \mbbH \in \mathcal{T}$, then $\mbbG \ast \mbbH \in \mathcal{T}$.
		\item If $\mbbG \in \mathcal{T}$, then $\mbbG \wr_* \mbbS_n^+ \in \mathcal{T}$.
	\end{enumerate}
\end{theorem}

The proof is constructive and gives a general polynomial-time algorithm that can be used to explicitly calculate the quantum automorphism group of any tree.

The proof is organized as follows.
First, we prove that the center of a tree is preserved by quantum automorphisms and use this to reduce the problem to finding quantum automorphism groups of rooted trees.
Intuitively, operation $(ii)$ of \autoref{thm:main_theorem} then corresponds to the quantum symmetries of the disjoint union of non-isomorphic subtrees and operation $(iii)$ corresponds to the quantum symmetries of the disjoint union of isomorphic subtrees.

\section{Preliminaries}

We assume that all graphs are finite and simple.
The vertex set of a graph $X$ is denoted by $V(X)$ and the edge set by $E(X)$.
For a vertex $i \in V(X)$, let $N(i)$ denote the set of all neighbours of $i$.
For two subsets $A,B \subseteq V(X)$, let $E(A,B)$ be the set of edges in $E(X)$ with one endpoint in $A$ and the other in $B$, and write $e(A,B) := |E(A,B)|$.
For singleton sets, we write $E(i,B)$ instead of $E(\{i\},B)$ and so on.
For a subset $A \subseteq V(X)$ and a vertex $i \in A$, write $d_A(i) := e(i,A)$.
When $A = V(X)$, we simply write $d(i)$ instead of $d_A(i)$. We will refer to a vertex of degree one in any graph (not just trees) as a \emph{leaf}.
%A~coloring of $X$ is a mapping from $V(X)$ to an initial segment of $\mbbN$.

\begin{definition}
	\label{def:rooted-tree}
	%A \emph{rooted tree} is a tree $R$ with one of its vertices $r$ designated as its \emph{root}. All vertices other than the root have the same color, say $c_0$, while the root has a distinct color, say $c_1$ that does not appear in the rest of the graph.
	A \emph{rooted tree} is a tuple $(T,r)$ where $T$ is a tree and $r \in V(T)$ is a designated vertex, called the \emph{root}.
    A \emph{forest of rooted trees} is a forest in which every connected component is a rooted tree.
\end{definition}

\begin{definition}
	\label{def:colored-graph}
	A \emph{(vertex-)colored graph} is a tuple $(X,c)$ where $X$ is a graph and $c : V(X) \to C$ is a function that assigns to each vertex $x \in V(X)$ a color $c(x)$ from some set of colors $C$. The color classes partition the vertex set, but we do not require that adjacent vertices receive distinct colors.
	
	Every (uncolored) graph will be thought of as a colored graph where all vertices have the same color, and every forest of rooted trees will be thought of as a colored graph where all roots have color $c_1$ and all non-roots have color $c_0 \neq c_1$.
\end{definition}

\begin{definition}
	A \emph{compact quantum group} $\mathbb{G}$ is a tuple $(\mathcal{A},\Delta)$, where $\mathcal{A}$ is a unital $C^*$\nobreakdash-algebra and $\Delta: \mathcal{A} \to \mathcal{A} \otimes \mathcal{A}$ is a unital $*$\nobreakdash-homomorphism satisfying the following conditions:
	\begin{enumerate}
		\item $(\Delta \otimes id) \Delta = (id \otimes \Delta) \Delta$;
		\item $\Delta(\mathcal{A})(1 \otimes \mathcal{A}) = \Delta(\mathcal{A})(\mathcal{A} \otimes 1)$ is dense in $A \otimes A$. 
	\end{enumerate}    
    The map $\Delta$ is called the \emph{comultiplication} of $\mathbb{G}$.
\end{definition}

Let $G$ be a compact classical group and let $C(G)$ be the commutative $C^*$\nobreakdash-algebra of continuous complex-valued functions on $G$. Since $C(G) \otimes C(G) \cong C(G \times G)$, we may define $\Delta: C(G) \to C(G)\otimes C(G)$ by $\Delta(f)(g \times h) = f(gh)$.
Then, $(C(G), \Delta)$ is a compact quantum group. Moreover, all compact quantum groups $\mathbb{G}$ where the $C^*$\nobreakdash-algebra is commutative are of this type. In view of this example, it is customary to denote the $C^*$\nobreakdash-algebra $\mathcal{A}$ associated with a compact quantum group $\mathbb{G} = (\mathcal{A},\Delta)$ as $C(\mathbb{G})$. The \emph{trivial quantum group} $\one$ is the trivial group $1$, viewed as a quantum group. The next proposition allows us to introduce non-trivial examples of compact quantum groups.

\begin{definition}
   Let $C(\mathbb{G})$ be a unital C$^*$-algebra generated by $\{u_{ij}:i,j\in [n]\}$ that satisfy the relations defined by the matrices $[u_{ij}]_{i,j=1}^n$ and $[u_{ij}^*]_{i,j=1}^n$ being invertible. Assume that $\Delta: C(\mathbb{G}) \to C(\mathbb{G}) \otimes C(\mathbb{G})$ is a unital $*$-homomorphism satisfying \begin{equation}\label{eqdef:comult-map}
    \Delta(u_{ij}) = \sum_{k=1}^n u_{ik}\otimes u_{kj}, \qquad i,j\in [n].
\end{equation}
Then, $\mathbb{G} = (C(\mathbb{G}), \Delta)$ forms a compact quantum group (see {\cite[Proposition 1.1.4]{neshveyev_tuset}}). Such compact quantum groups are knows as \emph{compact matrix quantum groups}. They are usually specified by the pair $\mathbb{G} = (C(\mathbb{G}), u)$, as the comultiplication only depends on $u$. The matrix $u$ will be called the \emph{fundamental representation} of the compact matrix quantum group $\mathbb{G}$.
\end{definition}

We now define the compact matrix quantum groups that are of interest to us.

\begin{definition}
	\label{def:Sn+}
	The \emph{quantum symmetric group} $\mathbb{S}_n^+=(C(\mathbb{S}_n^+), u)$ is the compact matrix quantum group, where $C(\mathbb{S}_n^+)$ is the universal C*-algebra generated by the generators $u_{ij}$ satisfying
	\begin{equation}
		\begin{aligned}\label{eq:qpercon}
			u_{ij}^2 = u_{ij} = u_{ij}^* & \ \text{ for all } \ i,j =1,...,n\\
			\sum_{j=1}^n u_{ij} = \sum_{i=1}^n u_{ij} = 1 & \ \text{ for all } \ i, j =1,...,n.\\
		\end{aligned}
	\end{equation} 
 A matrix satisfying the conditions in \autoref{eq:qpercon} is called a \emph{magic unitary}. 
\end{definition}

\begin{definition}
    Let $\mathbb{G} = (C(\mathbb{G}), \Delta_{\mathbb{G}})$ and $\mathbb{H} = (C(\mathbb{H}), \Delta_{\mathbb{H}})$ be compact quantum groups. We say that $\mathbb{G}$ is a \emph{quantum subgroup} of $\mathbb{H}$, if there is a surjective $\ast$-homomorphism $\phi: C(\mathbb{H}) \to C(\mathbb{G})$ such that $\Delta_{\mathbb{G}}\circ \phi = (\phi \otimes \phi) \circ \Delta_{\mathbb{H}}$. Quantum subgroups of the quantum symmetric group $S_n^+$ are known as \emph{quantum permutation groups}. 
\end{definition}

If $\mathbb{G} = (C(\mathbb{G}), u)$ is any compact matrix quantum group where $u \in M_n(C(\mathbb{G}))$ is a magic unitary, we see from the universal property of $S_n^+$ that $\mathbb{G}$ is a quantum permutation group.

\begin{definition}
	\label{def:Qut}
	We give the definitions of quantum automorphism groups of various kinds of graphs:
	\begin{enumerate}[label = (\Roman*)]
		\item The \emph{quantum automorphism group $\Qut(X)$ of a graph $X$} is the compact matrix quantum group $(C(\Qut(X)), u)$, where $C(\Qut(X))$ is the universal $C^*$-algebra with generators $u_{ij}$, for $i,j \in V(X)$, and relations
		\begin{center}
			\begin{equation}\label{eq:qg_def}
				\begin{aligned}
					u_{ij} = u_{ij}^*  &= u_{ij}^2, &  i,j \in V(X), & \\
					\sum_{k} u_{ik} = 1 &= \sum_{k} u_{kj}, & i,j \in V(X), & \\
					u_{ij}u_{kl} & = 0 & \text{ if } ik\in E, jl\notin E \text{ or vice versa}. &
				\end{aligned}
			\end{equation}
		\end{center}
		Assuming the first two conditions of~\eqref{eq:qg_def}, the third condition is equivalent to $A_X  u = u A_X $, which can be rewritten as
		\begin{equation}
			\sum_{kj \in E} u_{ik} = \sum_{ik \in E} u_{kj},
		\end{equation}
		for each $i,j$.
		
		\item{The \emph{quantum automorphism group $\Qut_c(X)$ of a graph $X$ with coloring $c$} is the compact matrix quantum group that is obtained by taking the quotient of $C(\Qut(X))$ by the relations $u_{ij} = 0$ if $c(i) \neq c(j)$. This is equivalent to adding the relations $D^aU = uD^a$ for every vertex color $a$, where $D^a$ is a diagonal matrix with $D^a_{ii} = 1$ if $i$ has color $a$ and $D^a_{ii} = 0$ otherwise.}
		
		\item{ The \emph{quantum automorphism group of a rooted tree $R$} is $\Qut_c(R)$. We shall denote $\Qut_c(R)$ as $\Qut_r(R)$ for convenience.}
	\end{enumerate}
\end{definition}

%Let $\tree$ denote the class of all trees.
%For a class of graphs $\mathcal{C}$, let $\Qut(\mathcal{C})$ be the class of quantum groups that can be realised as the quantum automorphism group of a graph in $\mathcal{C}$.

\begin{definition}
	Let $\mathbb{G} = (C(\mathbb{G}), u)$ and $H = (C(\mathbb{H}), v)$ be compact matrix quantum groups. Then, their \emph{free product} $\mathbb{G} \ast \mathbb{H}$ is defined as the compact matrix quantum group $(C(\mathbb{G}) \ast C(\mathbb{H}), u \oplus v)$, where $C(\mathbb{G}) \ast C(\mathbb{H})$ is the universal C*-algebra with generators $u_{ij}$ and $v_{kl}$ such that $u_{ij}$ and $v_{kl}$ satisfy the relations of $\mathbb{G}$ and $\mathbb{H}$ respectively, in addition to the relation $1_{\mathbb{G}} = 1_{\mathbb{H}}$.
\end{definition}

\begin{definition}
	\label{def:free-wreath-product}
	Let $\mathbb{G}=(C(\mathbb{G}), u)$ and $\mathbb{H}=(C(\mathbb{H}), v)$ be two quantum permutation groups such that their fundamental representations are magic unitaries. The \emph{free wreath product of $\mathbb{G}$ and $\mathbb{H}$} is the quantum permutation group
	$$C(\mathbb{G} \wr_* \mathbb{H}) = (C(\mathbb{G})^{*n} \ast C(\mathbb{H}))/ \langle [u_{ij}^{(a)}, v_{ab}] = 0 \rangle$$
	where $n \times n$ is the size of matrix $v$. The fundamental representation of $\mathbb{G} \wr_* \mathbb{H}$ is given by the magic unitary
	$$w_{ia, jb} = u_{ij}^{(a)}v_{ab}.$$
\end{definition}

A closely related concept to quantum automorphism groups of graphs is quantum isomorphism of graphs. This notion was originally defined in terms of quantum strategies for the so-called \emph{isomorphism game}~\cite{atserias2019quantum}. However, for this work the following equivalent definition from~\cite{Lupini-Mancinska-Roberson} is better suited\footnote{In both~\cite{atserias2019quantum} and~\cite{Lupini-Mancinska-Roberson} only uncolored graphs were considered, but adapting their results to include vertex colors is trivial.}:

\begin{definition}
    Let $X$ and $Y$ be (possibly colored) graphs. We say that $X$ and $Y$ are \emph{quantum isomorphic}, and write $X \cong_q Y$ if there is a nonzero unital $C^*$-algebra $\mathcal{A}$ and a magic unitary $u = [u_{xy}]_{x \in V(X), y \in V(Y)}$ with entries from $\mathcal{A}$ such that $A_Xu = uA_Y$ and $u_{xy} = 0$ if $c(x) \ne c(y)$.
\end{definition}

It is clear that if two graphs are isomorphic, then they are quantum isomorphic (the magic unitary in the definition above can be taken to be the permutation matrix encoding the isomorphism). But it is a nontrivial result that there are non-isomorphic graphs that are quantum isomorphic~\cite{atserias2019quantum}. However, as we will see in Section~\ref{sec:orbsncolors}, isomorphism and quantum isomorphism coincide for forests.

\section{Quantum orbits and color refinement}\label{sec:orbsncolors}

\begin{definition}
    Let $\mathbb{G} = (C(\mathbb{G}), u)$ be a quantum permutation group. We define a relation $\sim$ on $[n]$ by $i \sim j$ if and only if $u_{ij} \neq 0$. In~\cite{banica2018modeling} and~\cite{Lupini-Mancinska-Roberson} it was shown that $\sim$ is an equivalence relation. The partitions of $[n]$ induced by the equivalence relation $\sim$ are known as the \emph{orbits} or \emph{quantum orbits} of $\mathbb{G}$.
\end{definition}

\begin{remark}
    Note that by definition of the quantum automorphism group of a colored graph, vertices of different colors must be in different quantum orbits. In other words, the color classes of a colored graph are unions of quantum orbits.
\end{remark}

It follows from Theorem 4.5 of~\cite{Lupini-Mancinska-Roberson} and Corollary 7.5 of~\cite{atserias2019quantum} that computing the orbits of a quantum permutation group is undecidable in general. However, there are known and efficiently computable partitions of graphs that are guaranteed to be coarse-grainings of the partition into quantum orbits. Such partitions are useful in the study of quantum automorphism groups as they provide necessary conditions for vertices of a graph $X$ to be in the same orbit of $\Qut(X)$. The simplest such method is known as \emph{color-refinement} (see for instance~\cite{grohe2017color}). This begins by labeling every vertex of a colored graph by its color. The labeling is then iteratively refined by appending to each vertex's label the multiset of the labels of its neighbors. This iterative refinement necessarily stabilizes (i.e.~the partition induced by the labeling stops changing) after at most $|V(X)|$ steps. This final partition is the output of the algorithm. Color-refinement is also sometimes referred to as the 1-dimensional Weisfeiler-Leman algorithm. For uncolored graphs, the algorithm begins by labeling every vertex the same. Note then that after one iteration every vertex is labeled by its degree.

In~\cite{Lupini-Mancinska-Roberson}, it was shown that the partition of $V(X) \times V(X)$ found by the 2-dimensional Weisfeiler-Leman algorithm is always a (possibly trivial) coarse-graining of the orbits of $\Qut(X)$ on $V(X) \times V(X)$\footnote{Orbits of $\Qut(X)$ on $V(X) \times V(X)$ are not needed for this work, so we refrain from giving a definition. Thus, outside of this paragraph our orbits will always refer to orbits on $V(X)$.}. Every orbit of $\Qut(X)$ on $V(X) \times V(X)$ is either contained in or disjoint from $D(X) := \{(i,i): i \in V(X)\}$, and the orbits of the former type are precisely the orbits of $\Qut(X)$ on $V(X)$ (under the identification $i \mapsto (i,i)$). Additionally, every element of the partition of $V(X) \times V(X)$ found by the 2-dimensional Weisfeiler-Leman algorithm is also either contained in or disjoint from $D(X)$, and moreover the parts of the former type are form a partition of $V(X)$ which is a (possibly trivial) refinement of the partition found by color-refinement. It follows immediately from this that the partition found by color-refinement is a (possibly trivial) coarse-graining of the partition of $V(X)$ into the orbits of $\Qut(X)$. We formally state this in the lemma below:

\begin{lemma}\label{lem:colorrefineorbits}
    Let $X$ be a (possibly colored) graph and suppose that $P_1, \ldots, P_r \subseteq V(X)$ is the partition of $V(X)$ found by color-refinement. If $i \in P_l$ and $j \in P_k$ for $l \ne k$, then $u_{ij} = 0$, i.e., $i$ and $j$ are in different orbits of $\Qut_c(X)$.
\end{lemma}

\begin{remark}\label{rem:degsorbs}
    Since vertices of a (possibly colored) graph $X$ of different degrees already receive distinct labels after one iteration of color-refinement, such vertices are necessarily in different parts of the partition found by color-refinement and therefore in different orbits of $\Qut(X)$.
\end{remark}

Color-refinement can also be used as an isomorphism test. Given graphs $X$ and $Y$, one can run color-refinement on their disjoint union. If the resulting partition has a part that contains a different number of vertices of $X$ and $Y$, then we can conclude that $X$ and $Y$ are not isomorphic. If every part of the resulting partition contains an equal number of vertices from $X$ and $Y$, then the graphs are said to be \emph{fractionally isomorphic} and this has many equivalent formulations (e.g., the existence of a doubly stochastic matrix $D$ satisfying $A_XD = DA_Y$). It follows from Theorem~4.5 from~\cite{atserias2019quantum} that if two graphs are quantum isomorphic, then they are fractionally isomorphic. Additionally, it is known~\cite{Immerman-Lander-1990} that any two trees are fractionally isomorphic if and only if they are isomorphic. Therefore we have the following:

\begin{lemma}\label{lem:qiso2iso}
    If $F$ and $F'$ are trees, then they are quantum isomorphic if and only if they are isomorphic.
\end{lemma}

In addition to the above, we will need to use the following property of orbits of $\Qut_c(X)$ which as far as we are aware is a new result.

\begin{lemma}\label{lem:inducedorbits}
    Let $X$ be a (possibly colored) graph and suppose that $S \subseteq V(X)$ is a union of orbits of $\Qut_c(X)$. If $Y$ is the (possibly colored) subgraph of $X$ induced by $S$, then the orbits of $\Qut_c(Y)$ are unions of orbits of $X$.
\end{lemma}
\begin{proof}
    We will first prove the uncolored case and then remark on the additions needed when colors are present. Let $u$ and $v$ be the fundamental representations of $\Qut(X)$ and $\Qut(Y)$ respectively. Since $S$ is the union of orbits of $\Qut(X)$, the fundamental representation $u$ can be written as
    \[u = \begin{pmatrix}\hat{v} & 0 \\ 0 & \hat{u}\end{pmatrix}\]
    where $\hat{v}$ is a magic unitary indexed by the elements of $S$. We will show that $A_Y \hat{v} = \hat{v} A_Y$. Let $D$ be the diagonal matrix such that $D_{ii} = 1$ if $i \in S$ and $D_{ii} = 0$ otherwise. Then $Du = uD$, and therefore $DA_XDu = uDA_XD$. Of course,
    \[DA_XD = \begin{pmatrix}A_Y & 0 \\ 0 & 0\end{pmatrix},\]
    and therefore $DA_XDu = uDA_XD$ is equivalent to $A_Y \hat{v} = \hat{v} A_Y$. Thus we have proven our claim.

    Now by universality of $C(\Qut(Y))$, we have that there is a $*$-homomorphism $\varphi$ from $C(\Qut(Y))$ to $C(\Qut(X))$ such that $\varphi(v_{ij}) = \hat{v}_{ij}$. Thus if $v_{ij} = 0$, then $\hat{v}_{ij} = 0$, i.e., if $i,j \in V(Y)$ are in different orbits of $\Qut(Y)$, then they are in different orbits of $\Qut(X)$. This is the contrapositive of the lemma statement.

    If the graph $X$ is vertex colored, then we must additionally show that for any color $a$ appearing in $Y$, the $V(Y)$-indexed diagonal matrix $D^a$ that indicates whether a vertex of $Y$ is colored $a$ commutes with $\hat{v}$. However, if $\hat{D}^a$ is the similarly defined $V(X)$-indexed diagonal matrix, then $D\hat{D}^aD = D^a$ and the argument works the same as for $A_Y$.
\end{proof}

\section{Modifications that do not change the quantum automorphism group}

In this section, we show that the quantum automorphism group of a colored graph is invariant under certain modifications to the graph.
This will greatly simplify the proofs in the remainder of this paper, since most proofs can be reduced to a combination of the modifications from this section.
The main result in this section is the following lemma.

\begin{lemma}
	\label{lem:modifications}
	Let $X$ be a colored graph.
	\begin{enumerate}[label=\textup(\Roman*\textup)]
		\item\label{itm:mod:clique} Let $S \subseteq V(X)$ be an independent set which is a union of color classes.
		Then adding $\binom{|S|}{2}$ edges to $X$, one between every pair of distinct vertices in $S$, does not change the quantum automorphism group.
		
		\item\label{itm:mod:ST} Let $S,T \subseteq V(X)$ be disjoint vertex sets such that $S \cup T$ is an independent set and each of $S$ and $T$ is a union of color classes.
		Then adding $|S| \times |T|$ edges to $X$, one from every $s \in S$ to every $t \in T$, does not change the quantum automorphism group.
		
		\item\label{itm:mod:recolor} Let $S \subseteq V(X)$ be a monochromatic vertex set that is a union of quantum orbits of $X$.
		Then changing the color of $S$ to a new color \textup(that does not occur elsewhere\textup) does not change the quantum automorphism group.
		
		\item\label{itm:mod:isolated} Adding an isolated vertex in a new color \textup(that does not occur elsewhere\textup) does not change the quantum automorphism group.
	\end{enumerate}
\end{lemma}
\begin{proof}
	\begin{my-enumerate}
		\item Let $Y = X \cup \binom{S}{2}$ be the colored graph obtained from $X$ by adding all possible edges between vertices in $S$.
		Let $u = (u_{ij})_{i,j \in V(X)}$ be a magic unitary in an arbitrary $C^*$-algebra that respects the color classes of $X$ (i.e. $u_{ij} = 0$ whenever $c(i) \neq c(j)$).
		We prove that $u$ satisfies the relations of $\Qut_c(X)$ if and only if it satisfies the relations of $\Qut_c(Y)$.
		With respect to the ordered partition $V(X) = V(Y) = S \sqcup R$ with $R = V(X) \setminus S$, we see that $u$ and $A_Y - A_X$ are in block diagonal form
		\[ u = \begin{pmatrix}
			u_S & 0 \\
			0 & u_R
		\end{pmatrix},
		\qquad
		A_Y - A_X = \begin{pmatrix}
			J_S - I_S & 0 \\
			0 & 0
		\end{pmatrix}, \]
		where $u_S$ and $u_R$ denote the restrictions of $u$ to $S \times S$ and $R \times R$, and where $J_S$ and $I_S$ denote the all-ones matrix and the identity matrix on $S \times S$, respectively.
		Since rows and columns in $u_S$ sum to $1$, we have $u_SJ_S = J_S = J_Su_S$, and therefore
		\[ \begin{pmatrix}
			u_S & 0 \\
			0 & u_R
		\end{pmatrix}\begin{pmatrix}
			J_S - I_S & 0 \\
			0 & 0
		\end{pmatrix} = \begin{pmatrix}
			J_S - u_S & 0 \\
			0 & 0
		\end{pmatrix} = \begin{pmatrix}
			J_S - I_S & 0 \\
			0 & 0
		\end{pmatrix}\begin{pmatrix}
			u_S & 0 \\
			0 & u_R
		\end{pmatrix}. \]
		This shows that $u$ commutes with $A_{Y} - A_X$.
		Therefore $u$ commutes with $A_X$ if and only it commutes with $A_{Y}$.
		The color classes remain the same, so we conclude that $u$ satisfies the relations of $\Qut_c(X)$ if and only if it satisfies the relations of $\Qut_c(Y)$.
		By universality of the respective $C^*$\nobreakdash-algebras, we get natural $*$\nobreakdash-homomorphisms $\phi : C(\Qut_c(X)) \to C(\Qut_c(Y))$ and $\psi : C(\Qut_c(Y)) \to C(\Qut_c(X))$ that map the fundamental representation of one to the fundamental representation of the other.
		Clearly $\phi$ and $\psi$ are each other's inverses and preserve the comultiplication, so we have $\Qut_c(X) \cong \Qut_c(Y)$. 
		
		\item Let $Y'$ be the colored graph obtained from $X$ by adding all possible edges between $S$ and $T$.
		%Borrowing notation from \ref{itm:mod:clique}, note that
		Since $S$ and $T$ are independent sets in $Y'$, it follows from \ref{itm:mod:clique} that $\Qut_c(Y') \cong \Qut_c(Y' \cup \binom{S}{2} \cup \binom{T}{2})$.
		But we have $Y' \cup \binom{S}{2} \cup \binom{T}{2} = X \cup \binom{S \cup T}{2}$, so another application of \ref{itm:mod:clique} shows that
		\[ \Qut_c(Y') \cong \Qut_c(Y' \cup \textstyle\binom{S}{2} \cup \textstyle\binom{T}{2}) = \Qut_c(X \cup \textstyle\binom{S \cup T}{2}) \cong \Qut_c(X). \]
		
		\item Let $Y''$ be the colored graph obtained from $X$ by recoloring the vertices of $S$ to a new color (that does not occur elsewhere).
		Let $R \subseteq V(X) \setminus S$ be the remainder of the color class containing $S$, and let $u = (u_{ij})_{i,j \in V(X)}$ and $v'' = (v_{ij}'')_{i,j \in V(Y'')}$ denote the fundamental representations of $\Qut_c(X)$ and $\Qut_c(Y'')$, respectively.
		The only difference between $X$ and $Y''$ is that the latter splits the color class $S \cup R$ of $X$ into separate color classes $S$ and $R$, so the set of relations defining $C(\Qut_c(Y''))$ is the union of the set of relations defining $C(\Qut_c(X))$ and the relations $\{v_{sr}'' = v_{rs}'' = 0 \, : \, s \in S, r \in R\}$.
		Since $S$ is a union of quantum orbits of $X$, the latter (added) relations are also satisfied by $u$, so $u$ satisfies all the relations of $C(\Qut_c(Y''))$.
		Conversely, since the relations defining $C(\Qut_c(Y''))$ contain the relations defining $C(\Qut_c(X))$, clearly $v''$ satisfies all the relations of $C(\Qut_c(X))$.
		By universality of the respective $C^*$\nobreakdash-algebras, we get natural $*$\nobreakdash-homomorphisms $\phi : C(\Qut_c(X)) \to C(\Qut_c(Y''))$ and $\psi : C(\Qut_c(Y'')) \to C(\Qut_c(X))$ that map the fundamental representation of one to the fundamental representation of the other.
		Clearly $\phi$ and $\psi$ are each other's inverses and preserve the comultiplication, so we have $\Qut_c(X) \cong \Qut_c(Y'')$.
		
		\item Let $Y'''$ be the colored graph obtained from the construction in \ref{itm:mod:isolated}, let $v_0 \in V(Y''') \setminus V(X)$ denote the isolated vertex added to $X$, and let $u = (u_{ij})_{i,j \in V(X)}$ and $v''' = (v_{ij}''')_{i,j \in V(Y''')}$ denote the fundamental representations of $\Qut_c(X)$ and $\Qut_c(Y''')$, respectively.
		Then clearly the magic unitary $u' = (u_{ij}')_{i,j \in V(Y''')}$ defined by
		\[ u_{ij}' = \begin{cases}
			1,&\quad \text{if $i = j = v_0$}; \\
			u_{ij},&\quad\text{if $i,j \in V(X)$}; \\
			0,&\quad\text{otherwise};
		\end{cases} \]
		satisfies the relations of $\Qut_c(Y''')$, and the restriction $v$ of $v'''$ to $V(X) \times V(X)$ satisfies the relations of $\Qut_c(X)$.
		By universality of the respective $C^*$\nobreakdash-algebras, we get natural $*$\nobreakdash-homomorphisms $\phi : C(\Qut_c(X)) \to C(\Qut_c(Y'''))$ and $\psi : C(\Qut_c(Y''')) \to C(\Qut_c(X))$ that map $u$ to $v$ and $v'''$ to $u'$, respectively.
		Clearly $\phi$ and $\psi$ are each other's inverses, so we have $C(\Qut_c(X)) \cong C(\Qut_c(Y'''))$ as $C^*$\nobreakdash-algebras.
		To see that this isomorphism also preserves the comultiplication, note that for all $i,j \in V(X)$ we have
		\[ ((\phi \tensor \phi) \circ \Delta_X)(u_{ij}) = \sum_{k \in V(X)} v_{ik}''' \tensor v_{kj}''' = \sum_{k \in V(Y''')} v_{ik}''' \tensor v_{kj}''' = (\Delta_{Y'''} \circ \phi)(u_{ij}), \]
		where the middle equality uses that $v_{iv_0}''' = v_{v_0j}''' = 0$ for all $i,j \in V(X)$.
		Since the set $\{u_{ij} \, : \, i,j, \in V(X)\}$ generates $C(\Qut_c(X))$, it follows that $(\phi \tensor \phi) \circ \Delta_X = \Delta_{Y'''} \circ \phi$.
		\qedhere
	\end{my-enumerate}
\end{proof}

\begin{remark}
	In \myautoref{lem:modifications}{itm:mod:clique} and \ref{itm:mod:ST}, the assumption that $S$ or $S \cup T$ is an independent set can be omitted if one is willing to work with multigraphs.
	To do so, one has to define the quantum automorphism group of a multigraph in the same way as \autoref{def:Qut}, namely as the magic unitaries that preserve color classes and commute with the adjacency matrix (which can now contain entries other than $0$ and $1$).
	Note however that this might be considered to be the ``wrong'' definition of the quantum automorphism group of a multigraph, since it does not see (quantum) automorphisms that leave the vertices invariant but permute parallel edges.
	
	For simple graphs, \myautoref{lem:modifications}{itm:mod:ST} remains true if the condition ``$S \cup T$ is an independent set'' is weakened to ``there are no edges between $S$ and $T$'', by the same proof as above, but this proof now uses multigraphs in an intermediate step.
	We have no need for this more general version, so we have stated the result with the slightly stronger assumption (that $S \cup T$ is an independent set) to make sure all graphs in the intermediate steps remain simple.
\end{remark}

\section{Reduction to rooted trees}

Now we turn our attention to trees.
In this section, we show how to transform a tree into a rooted tree with the same quantum automorphism group.

Recall the following definitions.
\begin{definition}
	Let $X$ be a graph.
	The \emph{eccentricity} of a vertex $v \in V(X)$, denoted $\epsilon(v)$, is the maximum distance from $v$ to any other vertex in $X$.
	The \emph{radius} of $X$, denoted $r(X)$, is the minimum eccentricity of any vertex in $X$.
\end{definition}

\begin{definition}
	The \emph{\textup(Jordan\textup) center} of $X$, which we will denote by $Z(X)$, is the set of vertices of minimum eccentricity.
\end{definition}

It is well-known that the Jordan center of a tree coincides with the vertex set obtained by simultaneously removing all leaves from the graph and repeating this procedure until only a single vertex or edge remains.%
\footnote{In fact, Jordan's original definition of the center (see \cite[\S 4]{Jordan-center}) was based on this iterative procedure and not on the modern definition in terms of distances.
To prove that these notions are equivalent, note that the Jordan center of a tree $T$ with $|V(T)| \geq 3$ is equal to the Jordan center of the tree $T'$ obtained from $T$ by removing all leaves.}
It follows that the Jordan center of a tree consists of either a single vertex or two adjacent vertices.
In the latter case, we refer to the edge between the two vertices in the center as the \emph{central edge}.

In the 19th century, Jordan \cite{Jordan-center} already recognized that the center of a tree must be preserved by every automorphism.
We extend this to quantum automorphisms, and use it as the basis for our reduction from trees to rooted trees.
For this we use the following construction.

\begin{definition}
	Let $T$ be a tree.
	The \emph{rootification} of $T$ is the rooted tree $(T_r,r)$ obtained from $T$ in the following way:
	\begin{itemize}
		\item if $|Z(T)| = 1$, set $T_r := T$ and let $r$ be the central vertex of $T$;
		\item if $|Z(T)| = 2$, let $T_r$ be the tree obtained from $T$ by subdividing the central edge of $T$, and let $r \in V(T_r)$ be the new vertex thus created.
	\end{itemize}
\end{definition}

Below, we will show that $\Qut(T) \cong \Qut_r(T_r)$ for any tree $T$.
The key to this result is the following lemma.

\begin{lemma}
	\label{lem:leaf-removal-orbits}
	Let $X$ be a graph, and let $S \subseteq V(X)$ be a vertex set obtained by simultaneously deleting all leaves and repeating this process any number of times.
	Then $S$ is a union of quantum orbits.
\end{lemma}
\begin{proof}
Let $X_1$ be the graph obtained from $X$ by deleting all of its leaves. In other words, $X_1$ is the subgraph of $X$ induced by the set $U = \{v \in V(X) :  d(v) \ne 1\}$. By \autoref{rem:degsorbs}, $U$ is a union of quantum orbits of $X$ and therefore by \autoref{lem:inducedorbits} the quantum orbits of $X_1$ are unions of the quantum orbits of $X$.

Now, if $X_2$ is the graph obtained from $X_1$ by removing all of its leaves, then by the same reasoning as above $V(X_2)$ is a union of quantum orbits of $X_1$, and therefore the union of quantum orbits of $X$. Iterating this argument proves the lemma.
%	Let $V(X) = A_0 \supseteq A_1 \supseteq A_2 \supseteq \cdots \supseteq A_k = S$ denote the vertex sets obtained by repeatedly removing all leaves; that is, $A_{i+1} := \{v \in A_i \, : \, d_{A_i}(v) \neq 1\}$.
%	We prove by induction on $i$ that $A_i$ is a union of quantum orbits for all $i$.
%	For $i = 0$ this is clear.
%	Now suppose that $A_i$ is a union of quantum orbits.
%	Let $B_{i+1} := \{ v \in A_i \, : \, d_{A_i}(v) = 1\}$, so that $A_i = A_{i+1} \sqcup B_{i+1}$.
%	Then it follows from \autoref{lem:local-degree} that $u_{vw} = 0$ whenever $v \in A_{i+1}$ and $w \in B_{i+1}$, so it follows that each of $A_{i+1}$ and $B_{i+1}$ is also a union of quantum orbits.
	
%	By induction, it follows that $A_i$ is a union of quantum orbits for all $i$.
%	In particular, this is true for $S = A_k$.
\end{proof}

\begin{corollary}
	\label{cor:center-orbits}
	Let $T$ be a tree.
	Then $Z(T)$ is a union of quantum orbits.
\end{corollary}

\begin{corollary}
	\label{cor:root-orbit}
	Let $T$ be a tree, and let $(T_r,r)$ be its rootification.
	Then $\{r\}$ is a quantum orbit of $T_r$.
\end{corollary}
\begin{proof}
    Since $Z(T_r) = \{r\}$, this follows from \autoref{cor:center-orbits}.
\end{proof}

\begin{corollary}
	\label{cor:rooted-unrooted}
	Let $T$ be a tree, and let $(T_r,r)$ be its rootification.
	Then $\Qut(T_r) \cong \Qut_r(T_r)$; that is, the quantum automorphism group of $T_r$ as a graph is isomorphic to the quantum automorphism group of $T_r$ as a rooted tree.
\end{corollary}
\begin{proof}
	This follows from \autoref{cor:root-orbit} and \myautoref{lem:modifications}{itm:mod:recolor}.
\end{proof}

We now come to the main result of this section:

\begin{proposition}
	\label{prop:rootification-isomorphism}
	Let $T$ be a tree, and let $(T_r,r)$ be the rootification of $T$.
	Then $\Qut(T) \cong \Qut_r(T_r)$.
\end{proposition}
\begin{proof}
	If $|Z(T)| = 1$, then we have $T_r = T$, so the result follows from \autoref{cor:rooted-unrooted}.
	Otherwise, $|Z(T)| = 2$, and we let $Z(T) = \{z_1,z_2\}$.
	Let $X_0 := T$ and $X_5 := T_r$.
	We transform $X_0$ into $X_5$ through the following sequence of modifications from \autoref{lem:modifications}, which is also depicted in \autoref{fig:rootification}.
	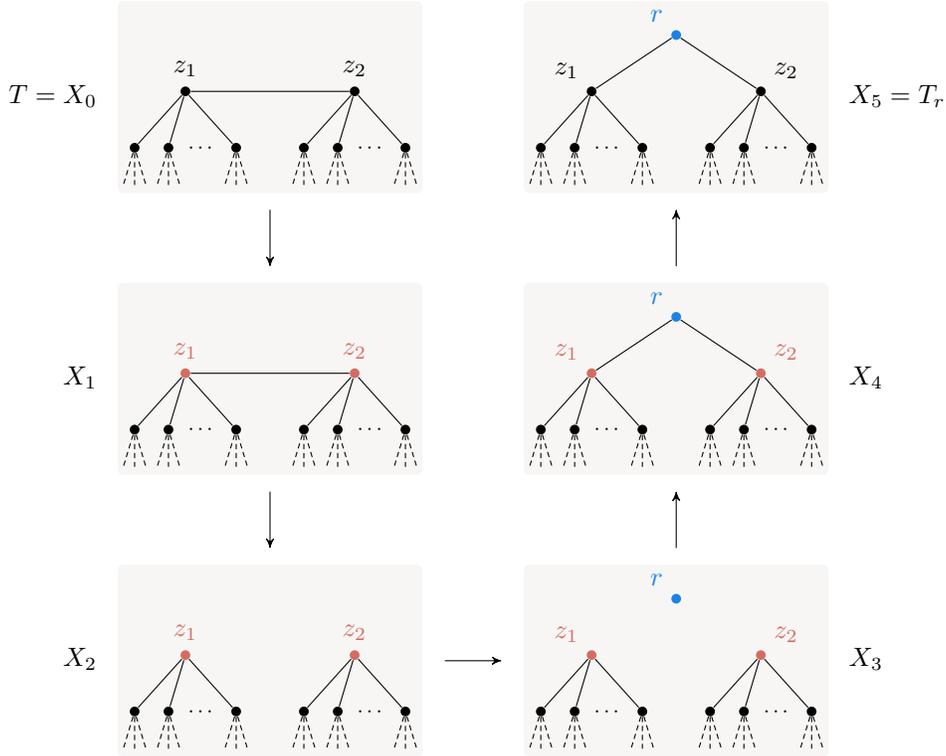
\begin{figure}[h!t]
		\centering
		\begin{tikzpicture}[scale=1.5]
			\begin{scope}
				% X0
				\fill[mybg,rounded corners=2pt] (-.6,-.9) rectangle ++(2.7,1.7);
				\node[anchor=base east] at (-.7,-.1) {$T = X_0$};
				\minitree{(0,0)}{l0}{c0}{above:$z_1$}
				\minitree{(1.5,0)}{r0}{c0}{above:$z_2$}
				\draw (Rl0) -- (Rr0);
				% downwards arrow from X0 to X1
				\draw[->] (.75,-1.05) -- ++(0,-.5);
				% X1
				\fill[mybg,rounded corners=2pt] (-.6,-3.4) rectangle ++(2.7,1.7);
				\node[anchor=base east] at (-.7,-2.6) {$X_1$};
				\minitree{(0,-2.5)}{l1}{c1}{above:$z_1$}
				\minitree{(1.5,-2.5)}{r1}{c1}{above:$z_2$}
				\draw (Rl1) -- (Rr1);
				% downwards arrow from X1 to X2
				\draw[->] (.75,-3.55) -- ++(0,-.5);
				% X2
				\fill[mybg,rounded corners=2pt] (-.6,-5.9) rectangle ++(2.7,1.7);
				\node[anchor=base east] at (-.7,-5.1) {$X_2$};
				\minitree{(0,-5)}{l2}{c1}{above:$z_1$}
				\minitree{(1.5,-5)}{r2}{c1}{above:$z_2$}
			\end{scope}
			% right arrow from X2 to X3
			\draw[->] (2.3,-5.05) -- ++(.5,0);
			\begin{scope}[xshift=3.6cm]
				% X5
				\fill[mybg,rounded corners=2pt] (-.6,-.9) rectangle ++(2.7,1.7);
				\node[anchor=base west] at (2.2,-.1) {$X_5 = T_r$};
				\minitree{(0,0)}{l5}{c0}{above left:$z_1$}
				\minitree{(1.5,0)}{r5}{c0}{above right:$z_2$}
				\draw[c2] (.75,.5) node[vertex,label=above left:$r$] (Ru5) {};
				\draw (Rr5) -- (Ru5) -- (Rl5);
				% upwards arrow from X4 to X5
				\draw[<-] (.75,-1.05) -- ++(0,-.5);
				% X4
				\fill[mybg,rounded corners=2pt] (-.6,-3.4) rectangle ++(2.7,1.7);
				\node[anchor=base west] at (2.2,-2.6) {$X_4$};
				\minitree{(0,-2.5)}{l4}{c1}{above left:$z_1$}
				\minitree{(1.5,-2.5)}{r4}{c1}{above right:$z_2$}
				\draw[c2] (.75,-2) node[vertex,label=above left:$r$] (Ru4) {};
				\draw (Rr4) -- (Ru4) -- (Rl4);
				% upwards arrow from X3 to X4
				\draw[<-] (.75,-3.55) -- ++(0,-.5);
				% X3
				\fill[mybg,rounded corners=2pt] (-.6,-5.9) rectangle ++(2.7,1.7);
				\node[anchor=base west] at (2.2,-5.1) {$X_3$};
				\minitree{(0,-5)}{l3}{c1}{above left:$z_1$}
				\minitree{(1.5,-5)}{r3}{c1}{above right:$z_2$}
				\draw[c2] (.75,-4.5) node[vertex,label=above left:$r$] (Ru3) {};
			\end{scope}
		\end{tikzpicture}
		\caption{Transforming any tree $T$ with Jordan center of size $2$ to its rootification $T_r$.}
		\label{fig:rootification}
	\end{figure}
	\begin{itemize}
		\item Let $X_1$ be the colored graph obtained from the (uncolored) graph $X_0 := T$ by coloring $V(T) \setminus Z(T)$ with the color $c_0$ and coloring $Z(T)$ in a different color $c_1 \neq c_0$.
		Then, by \autoref{cor:center-orbits} and \myautoref{lem:modifications}{itm:mod:recolor}, we have $\Qut(X_0) \cong \Qut_c(X_1)$.
		
		\item Let $X_2$ be the colored graph obtained from $X_1$ by removing the edge between $z_1$ and $z_2$.
		Then, by \myautoref{lem:modifications}{itm:mod:clique} (applied to $X_2$), we have $\Qut_c(X_1) \cong \Qut_c(X_2)$.
		
		\item Let $X_3$ be the colored graph obtained from $X_2$ by adding an isolated vertex $r$ in a new color $c_2 \neq c_0,c_1$.
		Then, by \myautoref{lem:modifications}{itm:mod:isolated}, we have $\Qut_c(X_2) \cong \Qut_c(X_3)$.
		
		\item Let $X_4$ be the colored graph obtained from $X_3$ by connecting all vertices of color $c_1$ (namely, $z_1$ and $z_2$) to all vertices of color $c_2$ (namely, $r$).
		Then, by \myautoref{lem:modifications}{itm:mod:ST}, we have $\Qut_c(X_3) \cong \Qut_c(X_4)$.
		
		\item Finally, let $X_5 := T_r$ be the rootification of $T$.
		Since $Z(T)$ can be obtained from $T$ by repeatedly removing all leaves, doing the same in $T_r$ shows that $\{z_1,z_2,r\}$ can be obtained from $T_r$ by repeatedly removing all leaves.
		Hence it follows from \autoref{lem:leaf-removal-orbits} that $\{z_1,z_2,r\}$ is a union of quantum orbits of $T_r$.
		Since $\{r\}$ is a quantum orbit by itself, it follows that $\{z_1,z_2\}$ is a union of quantum orbits of $T_r$.
		Since $X_4$ is (isomorphic to) the graph obtained from $X_5 = T_r$ by recoloring $\{z_1,z_2\}$ in a new color, it follows from \myautoref{lem:modifications}{itm:mod:recolor} (applied to $X_5$) that $\Qut_c(X_4) \cong \Qut_c(X_5)$.
	\end{itemize}
	%All in all, we have that $\Qut(T) = \Qut(X_0) \cong \Qut_c(X_1) \cong \Qut_c(X_2) \cong \Qut_c(X_3) \cong \Qut_c(X_4) \cong \Qut_c(X_5) = \Qut_r(T_r)$.
	All in all, we have $\Qut(T) = \Qut(X_0) \cong \Qut_c(X_1) \cong \cdots \cong \Qut_c(X_5) = \Qut_r(T_r)$.
	%All in all, we have that $\Qut(T) \cong \Qut_r(T_r)$.
\end{proof}

\section{Quantum Automorphisms of Trees}

\begin{lemma}
	\label{lem:tree-to-forest}
	Let $F$ be a forest of rooted trees, and let $\widetilde{F}$ be the rooted tree obtained by connecting the roots of the individual trees of $F$ to a single new root.
	Then, $\Qut_c(F) \cong \Qut_r(\widetilde{F})$.
\end{lemma}
\begin{proof}
	Let $R \subseteq V(F)$ denote the set of roots of $F$, and write $X_0 := F$ and $X_3 := \widetilde{F}$.
	Without loss of generality, assume that the non-root vertices in $F$ and $\widetilde{F}$ are colored with color $c_0$, the roots in $F$ are colored $c_1$, and the root in $\widetilde{F}$ is colored $c_2$ (with $c_0$, $c_1$ and $c_2$ distinct).
	We transform $X_0$ into $X_3$ using the following sequence of modifications from \autoref{lem:modifications}, which is also depicted in \autoref{fig:forest-to-tree}.
	\begin{figure}[h!t]
		\centering
		\begin{tikzpicture}[scale=1.5]
			\begin{scope}
				% X0
				\fill[mybg,rounded corners=2pt] (-.4,-.9) rectangle ++(2.8,1.7);
				\node[anchor=base east] at (-.5,-.1) {$F = X_0$};
				\nanotree{(0,0)}{l0}{c1}{}
				\nanotree{(.8,0)}{m0}{c1}{}
				\nanotree{(2,0)}{r0}{c1}{}
				\node at (1.4,0) {$\cdots$};
				% downwards arrow from X0 to X1
				\draw[->] (1,-1.05) -- ++(0,-.5);
				% X1
				\fill[mybg,rounded corners=2pt] (-.4,-3.4) rectangle ++(2.8,1.7);
				\node[anchor=base east] at (-.5,-2.6) {$X_1$};
				\nanotree{(0,-2.5)}{l1}{c1}{}
				\nanotree{(.8,-2.5)}{m1}{c1}{}
				\nanotree{(2,-2.5)}{r1}{c1}{}
				\node at (1.4,-2.5) {$\cdots$};
				\draw[c2] (1,-1.85) node[vertex,label=left:$r$] (Ru1) {};
			\end{scope}
			% right arrow from X2 to X3
			\draw[->] (2.6,-2.55) -- ++(.5,0);
			\begin{scope}[xshift=3.7cm]
				% X3
				\fill[mybg,rounded corners=2pt] (-.4,-.9) rectangle ++(2.8,1.7);
				\node[anchor=base west] at (2.5,-.1) {$X_3 = \widetilde{F}$};
				\nanotree{(0,0)}{l3}{c0}{}
				\nanotree{(.8,0)}{m3}{c0}{}
				\nanotree{(2,0)}{r3}{c0}{}
				\node at (1.4,0) {$\cdots$};
				\draw[c2] (1,.65) node[vertex,label=left:$r$] (Ru3) {};
				\draw (Ru3) -- (Rl3)
				      (Ru3) -- (Rm3)
				      (Ru3) -- (Rr3);
				% upwards arrow from X2 to X3
				\draw[<-] (1,-1.05) -- ++(0,-.5);
				% X2
				\fill[mybg,rounded corners=2pt] (-.4,-3.4) rectangle ++(2.8,1.7);
				\node[anchor=base west] at (2.5,-2.6) {$X_2$};
				\nanotree{(0,-2.5)}{l2}{c1}{}
				\nanotree{(.8,-2.5)}{m2}{c1}{}
				\nanotree{(2,-2.5)}{r2}{c1}{}
				\node at (1.4,-2.5) {$\cdots$};
				\draw[c2] (1,-1.85) node[vertex,label=left:$r$] (Ru2) {};
				\draw (Ru2) -- (Rl2)
				      (Ru2) -- (Rm2)
				      (Ru2) -- (Rr2);
			\end{scope}
		\end{tikzpicture}
		\caption{Transforming a forest of rooted trees to a single rooted tree.}
		\label{fig:forest-to-tree}
	\end{figure}
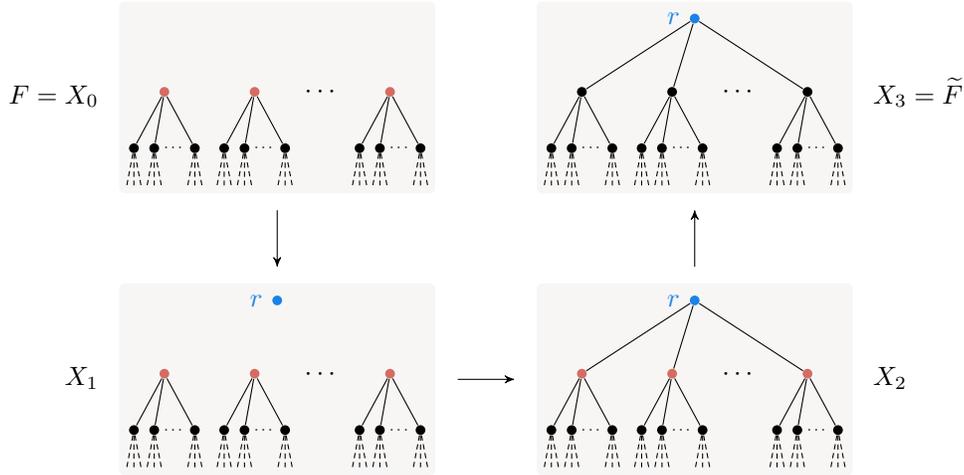
	\begin{itemize}
		\item Let $X_1$ be the graph obtained from $X_0 := F$ by adding an isolated vertex $r$ in a new color $c_2$ (that does not occur elsewhere in $X_0$).
		Then, by \myautoref{lem:modifications}{itm:mod:isolated}, we have $\Qut_c(X_0) \cong \Qut_c(X_1)$.
		
		\item Let $X_2$ be the graph obtained from $X_1$ by connecting all vertices of color $c_1$ (namely, the roots $R$ from $F$) to all vertices of color $c_2$ (namely, the newly added isolated vertex $r$).
		Then, by \myautoref{lem:modifications}{itm:mod:ST}, we have $\Qut_c(X_1) \cong \Qut_c(X_2)$.
		
		\item Finally, let $X_3 := \widetilde{F}$.
		Note that the color-refinement algorithm applied to $X_3 = \widetilde{F}$ distinguishes the roots from $F$ from the rest of the graph $\widetilde{F}$, since they are they only vertices with a neighbour with color $c_2$.
		Therefore $R$ is a union of quantum orbits of $\widetilde{F}$.
		Since $X_2$ is (isomorphic to) the graph obtained from $X_3 = \widetilde{F}$ by recoloring $R$ in a new color $c_1$, it follows from \myautoref{lem:modifications}{itm:mod:recolor} (applied to $X_3$), that $\Qut_c(X_2) \cong \Qut_c(X_3)$.
		\qedhere
	\end{itemize}
\end{proof}

%The following results are versions for vertex-colored graphs of results from \cite{Lupini-Mancinska-Roberson} and \cite{Schmidt-dissertation}. We skip the proof of both results as the proofs are nearly identical to the proofs of the original results.
%\begin{theorem}[{\cite[Theorem 4.5]{Lupini-Mancinska-Roberson}}]
	%Let $X$ and $Y$ be connected vertex colored graphs. Then, $X \cong_{qc} Y$ if and only if there exists $x \in V(X)$ and $y \in V(Y)$ that are in the same quantum orbit of $\Qut_c(X \sqcup Y)$.
%\end{theorem}

%\begin{lemma}[{\cite[Corollary 7.1.4]{Schmidt-dissertation}}]\label{lem:disjoint-union}
%	Let $X_1$ and $X_2$ be connected vertex colored graphs that are not quantum isomorphic. Then,
	%\begin{equation}
		%\Qut_c(X_1 \sqcup X_2) = \Qut_c(X_1) \ast \Qut_c(X_2)
	%\end{equation}
	%where $X_1 \sqcup X_2$ denotes the disjoint union of $X_1$ and $X_2$.
%\end{lemma}

\begin{lemma}[{\cite[Lemma 3.2.2]{Schmidt-dissertation}}]\label{lem:dist-from-vert}
	Let $X$ be a finite, undirected vertex-colored graph and let $(u_{ij})_{1 \leq i,j \leq n}$ be the generators of $C(Qut(X))$. If the distance between $i$ and $k$ is different to the distance between $j$ and $l$, then $u_{ij}u_{kl} = 0$.
\end{lemma}

\begin{lemma}
    Let $X$ be a (vertex-colored) graph with connected components $X_1, \dots, X_n$. If $u_{st}\neq 0$ for some $s\in V(X_i), t\in V(X_j), i\neq j$, in the fundamental representation of $\Qut(X)$, then $X_i$ and $X_j$ are quantum isomorphic.
\end{lemma}

\begin{proof}
    We adjust the proof of Theorem 4.4 in \cite{Lupini-Mancinska-Roberson} to our setting. Assume that we have $i,j$, $i\neq j$ such that there exist $s\in V(X_i), t\in V(X_j)$ with $u_{st}\neq 0$. We have $u_{xy}u_{x'y'}=0$ for $x,x'\in V(X_a)$, $y \in V(X_b)$, $y'\in V(X_c)$, $b\neq c$ by \autoref{lem:dist-from-vert}, since we know that there is a path between $x$ and $x'$ as they are in the same connected component but no path between $y$ and $y'$. 
    Define $p_{x,j}=\sum_{y \in V(X_j)} u_{xy}$, $x\in V(X)$. It holds that $p_{x,j}=p_{x',j}$ for $x,x' \in V(X_i)$ because of the following. We have
    \begin{align*}
        p_{x,j}(1-p_{x',j})=\left(\sum_{y \in V(X_j)} u_{xy}\right)\left(\sum_{y' \in V(X_k), k\neq j} u_{x'y'}\right)=\sum_{\substack{y \in V(X_j),\\y' \in V(X_k), k\neq j}} u_{xy}u_{x'y'}=0,
    \end{align*}
    since $u_{xy}u_{x'y'}=0$ for $x,x'\in V(X_i)$, $y \in V(X_j)$, $y'\in V(X_k)$. We get $(1-p_{x,j})p_{x',j}=0$ in the same way and therefore $p_{x,j}=p_{x,j}p_{x',j}=p_{x',j}$. We define $p_j:=p_{x,j}$. Similarly, we can define $p_{y,i}'=\sum_{x \in V(X_i)} u_{xy}$ and get $p_{y,i}'=p_{y',i}'$ for $y,y'\in V(X_j)$. We then define $p_i':=p_{y,i}'$.

    Note that $p_{j}\neq 0$ since we have $p_{j}\geq u_{st}\neq 0$ by assumption. It holds that
    \begin{align}
        |V(X_i)|p_{j}=\sum_{x \in V(X_i)}p_{x,j}=\sum_{\substack{x \in V(X_i)\\y \in V(X_j)}}u_{xy}=\sum_{y \in V(X_j)}p_{y,i}'=|V(X_j)|p_{i}'.\label{eq:samesum}
    \end{align}
    Then we have
    \begin{align*}
        p_{j}=p_{j}^2=\left(\frac{|V(X_j)|}{|V(X_i)|}\right)^2p_{i}'=\frac{|V(X_j)|}{|V(X_i)|}p_{j},
    \end{align*}
    and thus we must have $|V(X_i)| = |V(X_j)|$ since $p_{j} \ne 0$. Therefore, $p_{j}=p_{i}'$ by Equation \eqref{eq:samesum}. We define $p:=p_{j}$.

    Let $\mathcal{A}$ be the $C^*$-subalgebra of $C(\Qut(X))$ generated by the elements $u_{xy}$, $x \in V(X_i)$, $y\in V(X_j)$. Note that $\mathcal{A}$ is nonzero since $p \ne 0$ as noted above. Further, $p$ is the identity in $\mathcal{A}$, since 
    \begin{align*}
        u_{xy}p=u_{xy}p_{x,j}=u_{xy}\sum_{y'\in V(X_j)} u_{xy'}=u_{xy}
    \end{align*}
    and similarly $u_{xy}=pu_{xy}$. Moreover, we have $u_{xy}u_{x'y'}=0$ for $xx'\in E$, $yy'\notin E$ or vice versa, and $u_{xy}=0$ for $c(x)\neq c(y)$, since this was true in $C(\Qut(X))$. Thus, $\hat{u}=(u_{xy})_{x\in V(X_i), y \in V(X_j)}$ is a magic unitary with $A_{X_i}\hat{u}=\hat{u}A_{X_j}$. This means that $X_i$ and $X_j$ are quantum isomorphic. 
\end{proof}

The following lemma is a version of {\cite[Lemma 3.2.2]{Schmidt-dissertation}} for several, not necessarily connected graphs.

\begin{lemma}\label{lem:disjoint-union}
	Let $X_1, \dots, X_n$ be vertex colored graphs such that for any $i \ne j$, no connected component of $X_i$ is quantum isomorphic to a connected component of $X_j$. Then,
	\begin{equation}
		\Qut_c\left(\bigsqcup_{i=1}^n X_i\right) = \bigast_{i=1}^n\Qut_c(X_i) 
	\end{equation}
	where $\bigsqcup_{i=1}^n X_i$ denotes the disjoint union of $X_1, \dots, X_n$.
\end{lemma}

\begin{proof}
Define $X:=\bigsqcup_{i=1}^n X_i$ and let $A_{X}$ and $A_{X_i}$ be the adjacency matrices of the corresponding graphs. Label $A_{X}$ in the following way
\begin{align*}
A_{X} = \begin{pmatrix} A_{X_1}&0&0&\dots&0\\ 0&A_{X_2}&0&\dots&0\\ 0&0&A_{X_3}&\dots&0\\\vdots &\vdots&\vdots&\ddots&0\\0&0&0&0&A_{X_n} \\\end{pmatrix}.
\end{align*}
Since all connected components of $X_i$ and $X_j$, $i\neq j$ are not quantum isomorphic, we know by the previous lemma $u_{xy} =0$ for all $x \in V(X_i)$, $y \in V(X_j)$, where $u$ is the fundamental representation of $\Qut(X)$. This yields
\begin{align*}
u = \begin{pmatrix} u^{(1)}&0&0&\dots&0\\ 0&u^{(2)}&0&\dots&0\\ 0&0&u^{(3)}&\dots&0\\\vdots &\vdots&\vdots&\ddots&0\\0&0&0&0&u^{(n)} \\\end{pmatrix},
\end{align*}
where $u^{(i)} \in \mathrm{M}_{|V(X_i)|}(C(\Qut(X))$. Furthermore, we see that $uA_X = A_Xu$ is equivalent to $u^{(i)}A_{X_i} = A_{X_i}u^{(i)}$ for all $i$. Therefore, we get the desired surjective $*$-homomorphisms in both directions by using the respective universal properties. 
\end{proof}

%\commentJosse{Do we maybe want to state \autoref{lem:disjoint-union} more generally for a finite collection of connected graphs $X_1,\ldots,X_k$? After all, it seems that the lemma only works for connected graphs, so we cannot use it multiple times if we have more than two connected components (i.e.{} if the root of our rooted tree has more than two children). Is the lemma still true for more than two connected graphs?}

\begin{lemma}\label{lem:tree_to_rooted}
	For every rooted tree there exists a tree with an isomorphic quantum automorphism group.
\end{lemma}
\begin{proof}
	Let $T$ be a rooted tree.
	We construct a tree $T'$ with $\Qut_r(T) \cong \Qut(T')$.

	First, suppose that $T$ is a path rooted at one of its endpoints.
	Then, by induction, it follows from \autoref{lem:tree-to-forest} that $\Qut_r(T) \cong \Qut_r(K_1) = \one$.
	Since $\Qut(K_1) = \one$, the choice of $T' := K_1$ suffices.
	
	Now suppose that $T$ is not a path rooted at one of its endpoints.
	Let $X_1$ be the forest of rooted trees given by $X_1 := T \sqcup P$, where $P$ is a path of length $2|V(T)|$ rooted at one of its leaves.
	Let $X_2$ be the rooted tree formed from $X_1$ by adding a new vertex $r$ adjacent to only the roots of $T$ and $P$.
	Let $T'$ be the unrooted tree obtained from $X_2$ by ignoring the root.
	
	From the choice of $P$ and the definition of the centre, it follows that the centre of $T'$ lies on the path $P$.
	Hence, by \autoref{prop:rootification-isomorphism} we may conclude that $\Qut(T') = \Qut_r(T_r')$, where $T_r'$ is the rootification of $T'$.
	Let the $\widetilde{r}$ be the root of $T_r'$.
	We have that $\deg(\widetilde{r}) = 2$.
	Let $X_3$ and $X_4$ be the rooted trees obtained by deleting $\widetilde{r}$ and designating its neighbours as roots.
	Relabelling if necessary, we know that $X_3$ is a rooted tree that is obtained by adjoining a path to the root of $R$ and designating the other end as the root, and $X_4$ is a path.
	It follows from applying \autoref{lem:tree-to-forest} multiple times that $\Qut_r(X_3) = \Qut_r(T)$. Additionally, by the argument for the case where $T$ was a path rooted at one of its endpoints, we know that $\Qut_r(X_4) = \one$. 
	Finally, from \autoref{lem:tree-to-forest} and \autoref{lem:disjoint-union}, we may conclude that $\Qut(T') = \Qut_r(T_r') = \Qut_r(X_3) \ast \Qut_r(X_4) = \Qut_r(T)$.
	This finishes the proof.
\end{proof}

To prove our main result, we need the quantum automorphism group of $n$ copies of the same vertex-colored graph. This theorem was proven for uncolored graphs in \cite[Theorem 4.2]{Bichon-free-wreath-product}, for colored graphs it is a special case of \cite[Theorem 6.1]{Banica-Bichon-free-product}. 

\begin{theorem}[{\cite[Theorem 6.1]{Banica-Bichon-free-product}}]\label{thm:free-wreath-pro}
	Let $X$ be a connected vertex colored graph and $n \in \mathbb{N}$. Let $\bigsqcup_{i=1}^n X$ denote the disjoint union of $n$ copies of $X$, all with the same coloring. Then, $\Qut_c(\bigsqcup_{i=1}^n X) = \Qut_c(X) \wr_* \mathbb{S}_n^+$, where $\wr_*$ denotes the free wreath product.
\end{theorem}

\begin{proof}[Proof of \autoref{thm:main_theorem}]
	Let $\mathcal{S}$ be the family of compact quantum groups that is generated by \emph{(i)--(iii)}, and let $\mathcal{R}$ denote the set of all quantum automorphism groups of rooted trees. Recall that $\mathcal{T}$ denotes the set of all quantum automorphism groups of trees.
	It follows from \autoref{prop:rootification-isomorphism} and \autoref{lem:tree_to_rooted} that $\mathcal{T} = \mathcal{R}$.
	Hence, it is sufficient to show that $\mathcal{S} = \mathcal{R}$.
	
	First, we show that $\mathcal{S} \subseteq \mathcal{R}$. We do this by showing that $\one \in \mathcal{R}$ (which is trivial) and that $\mathcal{R}$ is closed under the operations described in \emph{(ii)} and \emph{(iii)} in the theorem statement. For \emph{(ii)}, we split it into the two cases of $\mathbb{G} \not\cong \mathbb{H}$ and $\mathbb{G} \cong \mathbb{H}$. For the former, consider any two non-isomorphic elements $\mathbb{G}, \mathbb{H} \in \mathcal{R}$.
	Let $T_{\mathbb{G}}$ and $T_{\mathbb{H}}$ be two rooted trees such that $\Qut(T_{\mathbb{G}}) \cong \mathbb{G}$ and $\Qut_r(T_{\mathbb{H}}) \cong \mathbb{H}$.
	Let us denote the rooted tree constructed by joining the roots of $T_{\mathbb{G}}$ and $T_{\mathbb{H}}$ to a single new root by $T_{\mathbb{G} \sqcup \mathbb{H}}$.
	It follows from \autoref{lem:tree-to-forest} that $\Qut_r(T_{\mathbb{G} \sqcup \mathbb{H}}) \cong \Qut_c(T_{\mathbb{G}} \sqcup T_{\mathbb{H}})$.
	Since $\mathbb{G}$ and $\mathbb{H}$ are distinct, $T_{\mathbb{G}}$ and $T_{\mathbb{H}}$ are not isomorphic, and hence they are also not quantum isomorphic by \autoref{lem:qiso2iso}.
	It now follows from \autoref{lem:disjoint-union} that $\Qut_r(T_{\mathbb{G} \sqcup \mathbb{H}}) \cong \mathbb{G} \ast \mathbb{H}$, so that $\mathbb{G} \ast \mathbb{H} \in \mathcal{R}$.
	
	Now we consider the $\mathbb{G} \cong \mathbb{H}$ case. Here we may assume that $\mathbb{G} = \mathbb{H}$ in fact. Let $\mathbb{G} \in \mathcal{R}$ be a compact quantum group, and let $T_{\mathbb{G}}$ be a rooted tree such that $\Qut_r(T_{\mathbb{G}}) \cong \mathbb{G}$.
	Let $\widetilde{T_{\mathbb{G}}}$ be the rooted tree constructed by joining the root of $T_{\mathbb{G}}$ to a new vertex, which is designated as the root of $\widetilde{T_{\mathbb{G}}}$.
	It follows from \autoref{lem:tree-to-forest} that $\Qut_r(\widetilde{T_{\mathbb{G}}}) \cong \mathbb{G}$.
	Since $T_{\mathbb{G}}$ and $\widetilde{T_{\mathbb{G}}}$ have a different number of vertices, they are not quantum isomorphic.
	Let $T$ be the rooted tree that is constructed by joining the roots of $T_{\mathbb{G}}$ and $\widetilde{T_\mathbb{G}}$ to a new vertex that is designated to be the root of $T$.
	Then, from \autoref{lem:tree-to-forest} and \autoref{lem:disjoint-union}, we have that $\Qut_r(T) \cong \Qut_c(T_{\mathbb{G}} \sqcup \widetilde{T_{\mathbb{G}}}) \cong \mathbb{G} \ast \mathbb{G}$, so that $\mathbb{G} \ast \mathbb{G} \in \mathcal{R}$.
	
	Let $\mathbb{G} \in \mathcal{R}$, and let $T_{\mathbb{G}}$ be a rooted tree such that $\Qut_r(T_{\mathbb{G}}) \cong \mathbb{G}$.
	It follows from \autoref{thm:free-wreath-pro} that $\Qut_c(\bigsqcup_{i=1}^n T_{\mathbb{G}}) \cong \mathbb{G} \wr_* \mathbb{S}_n^+$ for any $n \in \mathbb{N}$.
	Let $T'$ be the rooted tree that is constructed from the forest of rooted trees $\bigsqcup_{i=1}^n T_{\mathbb{G}}$ by joining the roots of individual copies of $T_{\mathbb{G}}$ to a single new root.
	Then, it follows from \autoref{lem:tree-to-forest} that $\Qut_r(T') \cong \Qut_c(\bigsqcup_{i=1}^n T_{\mathbb{G}}) \cong  \mathbb{G} \wr_* \mathbb{S}_n^+$, so that $\mathbb{G} \wr_* \mathbb{S}_n^+ \in \mathcal{R}$.
	Hence, we may conclude that any compact quantum group that is inductively constructed using \emph{(i)--(iii)} is the quantum automorphism group of a rooted tree, which implies that $\mathcal{S} \subseteq \mathcal{R}$.
	
	Now we show that $\mathcal{R} \subseteq \mathcal{S}$, i.e., that for any rooted tree $T$, $\Qut(T)$ can be constructed iteratively through \emph{(i)--(iii)}.
	Let $\widetilde{T}$ be the forest of rooted trees that is constructed by deleting the root of $T$, and designating the neighbours of $T$ as the roots of the trees that are formed.
	If we have $n$ equivalence classes of rooted trees (where two trees are equivalent if they are isomorphic), and the $i$th equivalence class has $m_i$ isomorphic copies of the rooted tree $T_i$, it follows from \autoref{thm:free-wreath-pro}, \autoref{lem:disjoint-union}, and \autoref{lem:tree-to-forest} that $\Qut_r(T) \cong \bigast_{i=1}^n (\Qut_r(T_i) \wr_* \mathbb{S}_{m_i}^+)$. 
 
	Now, we can iteratively apply the same deconstruction to the rooted trees $\{T_i\}_{i=1}^n$ until we end up with trees with only one vertex and no edges, whose quantum automorphism group is the trivial quantum group $\one$.
	Hence, $\Qut_r(T)$ can be constructed iteratively using \emph{(i)--(iii)}.
	We have now proved that $\mathcal{R} \subseteq \mathcal{S}$.
\end{proof}

\begin{figure}\label{fig:tree_example}
  \centering
  \includegraphics{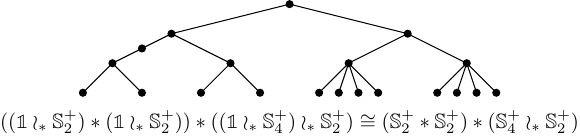}
  \caption{An example of a tree with its quantum automorphism group.}
\end{figure}

\paragraph{\textbf{Acknowledgments.}}\phantom{a}\newline
\noindent JvDdB, PK, DR, and PZ are supported by supported by the Carlsberg
    Foundation Young Researcher Fellowship CF21-0682 -- ``Quantum Graph Theory".
    
\noindent S.S.~was funded by the Deutsche Forschungsgemeinschaft (DFG, German Research Foundation) under Germany's Excellence Strategy - EXC 2092 CASA - 390781972.

\bibliographystyle{alphaurl}
\bibliography{quantum-automorphisms}

\end{document}